\documentclass{amsart}
\usepackage{amsthm,xcolor}
\usepackage{latexsym,amssymb,amsfonts,amsmath,graphicx, url}
\usepackage[vcentermath,enableskew]{youngtab}
\usepackage{color}
\usepackage[noadjust]{cite} 
\usepackage{tikz,hyperref}

\newtheorem{theorem}{Theorem}
\newtheorem{definition}[theorem]{Definition}

\newtheorem{corollary}[theorem]{Corollary}
\newcommand{\B}{\mathcal{B}}
\newcommand{\Newton}{\mathrm{Newton}}
\newcommand{\newword}[1]{\textbf{#1}}

\title[Schubert polynomials via  generalized permutahedra]{Schubert polynomials as integer point transforms of generalized permutahedra}
\author{Alex Fink}
\address{Alex Fink, School of Mathematical Sciences, Queen Mary University of London, UK, E1 4NS.
{a.fink@qmul.ac.uk}}

\author{Karola M\'esz\'aros}
\address{Karola M\'esz\'aros, Department of Mathematics, Cornell University, Ithaca NY 14853.  \newline{karola@math.cornell.edu}
}

\author{Avery St. Dizier}
\address{Avery St. Dizier, Department of Mathematics, Cornell University, Ithaca NY 14853.  \newline{ajs624@cornell.edu}
}

\thanks{Fink is partially supported by an Engineering and Physical Sciences Research Council grant (EP/M01245X/1),
and M\'esz\'aros by a National Science Foundation Grant (DMS 1501059).}

\date{\today}
\begin{document}

\begin{abstract} We show that the dual character of the flagged Weyl module of any diagram is a positively weighted integer point transform of a generalized permutahedron. In particular, 
Schubert  and key polynomials are positively weighted integer point transforms of generalized permutahedra. This implies several recent conjectures of Monical, Tokcan and Yong.  
\end{abstract}

\maketitle

\section{Introduction}
Schubert polynomials and key polynomials are classical objects in mathematics. Schubert polynomials, introduced  by Lascoux and Sch\"utzenberger in 1982 \cite{LS}, represent cohomology classes of Schubert cycles in flag varieties.  Key polynomials, also known as Demazure characters, are polynomials associated to compositions. Key polynomials were first introduced by Demazure for Weyl groups \cite{demazure}, and studied in the context of the symmetric group by Lascoux and Sch\"utzenberger in \cite{LS1,LS2}.

Beyond algebraic geometry, Schubert and key polynomials play an important role in algebraic combinatorics \cite{laddermoves,BJS,flaggedLRrule,FK1993, KM}. The second author and Escobar \cite{pipe1} showed that for permutations $1\pi'$ where $\pi'$ is dominant, Schubert polynomials are specializations of reduced forms in the subdivision algebra of flow and root polytopes. On the other hand, intimate connections of flow and root polytopes with generalized permutahedra have been exhibited by Postnikov \cite{beyond}, and more recently by the last two authors \cite{AK}.
These works imply that for permutations $1\pi'$ where $\pi'$ is dominant, 
the Schubert polynomial $\mathfrak{S}_{1\pi'}({\bf x})$ is equal to the  integer point transform of a generalized permutahedron \cite{AK}.

The main result of this paper proves that  the dual character of the flagged Weyl module of any diagram is a positively weighted integer point transform of a generalized permutahedron.  Since Schubert and key polynomials are dual characters of certain flagged Weyl modules, it follows that  the Newton polytope of any Schubert polynomial $\mathfrak{S}_\pi$ or key polynomial $\kappa_\alpha$ is a generalized permutahedron, and each of these polynomials is a sum over the lattice points in its Newton polytope with positive integral coefficients. 
 
After reviewing the necessary background, we prove our main theorem and draw corollaries about Schubert and key polynomials, confirming several recent conjectures of Monical, Tokcan and Yong \cite{MTY}. 
  
\section{Background}
\label{sec:bg} 

This section contains a collection of definitions of classical mathematical objects. Our basic notions are Schubert and key polynomials, Newton polytopes, generalized permutahedra, (Schubert) matroids, and flagged Weyl modules.

\subsection{Schubert  polynomials} 
The Schubert polynomial of the longest permutation $w_0=n \hspace{.1cm} n\!-\!1 \hspace{.1cm} \cdots \hspace{.1cm} 2 \hspace{.1cm} 1 \in S_n$ is 
\[\mathfrak{S}_{w_0}:=x_1^{n-1}x_2^{n-2}\cdots x_{n-1}.\]

For $w\in S_n$, $w\neq w_0$, there exists $i\in [n-1]$ such that $w(i)<w(i+1)$. 
For any such~$i$, the \newword{Schubert polynomial} $\mathfrak{S}_{w}$ is defined as 
\[\mathfrak{S}_{w}(x_1, \ldots, x_n):=\partial_i \mathfrak{S}_{ws_i}(x_1, \ldots, x_n),\] 
where $\partial_i$ is the $i$th divided difference operator
\[\partial_i (f):=\frac{f-s_if}{x_i-x_{i+1}} \mbox{ and }s_i=(i,i+1).\]
Since the $\partial_i$ satisfy the braid relations, the Schubert polynomials $\mathfrak{S}_{w}$ are well-defined. 

\subsection{Key polynomials}
A \newword{composition} $\alpha$ is a sequence of nonnegative integers $(\alpha_1,\alpha_2,\ldots)$ with $\sum_{k=1}^{\infty}\alpha_k<\infty$. If $\alpha$ is weakly decreasing, define the \newword{key polynomial} $\kappa_\alpha$ to be
\[\kappa_\alpha=x_1^{\alpha_1}x_2^{\alpha_2}\cdots. \]
Otherwise, set \[\kappa_\alpha = \partial_i\left( x_i\kappa_{\hat{\alpha}}\right)\mbox{ where } \hat{\alpha}=(\alpha_1,\ldots,\alpha_{i+1},\alpha_{i},\ldots) \mbox{ and }\alpha_i<\alpha_{i+1}. \]
It is an important fact due to Lascoux and Sch\"utzenberger \cite{LS1} that every Schubert polynomial is a sum of key polynomials. 

\subsection{Diagrams}
View $[n]^2$ as an $n$ by $n$ grid of boxes labeled $(i,j)$ in the same way as entries of an $n\times n$ matrix, with labels increasing as you move top to bottom along columns and left to right across rows from the upper-left corner. By a \textbf{diagram}, we mean a subset $D\subseteq [n]^2$, a collection of boxes in the $n\times n$ grid. Throughout this paper, we view $D$ as an ordered list of subsets $D=(D_1,D_2,\ldots,D_n)$ where for each $j$, $D_j=\{i:\, (i,j)\in D \}$ is the set of row indices of boxes of $D$ in column $j$. Two important classes of diagrams are Rothe diagrams and skyline diagrams.

\begin{definition}
	The \newword{Rothe diagram} of a permutation $\pi\in S_n$ is
	the collection of boxes $D(\pi)=\{(i,j):\, 1\leq i,j\leq n,\, \pi(i)>j,\, \pi^{-1}(j)>i \}$. 
	$D(\pi)$ can be visualized as the set of boxes left in the $n\times n$ grid 
	after you cross out all boxes weakly below or right of $(i,\pi(i))$ for each $i\in [n]$.
	Let $D(\pi)_j=\{i:\, (i,j)\in D(\pi) \}$ for each $j$, so $D(\pi)=(D(\pi)_1,\ldots, D(\pi)_n)$.
\end{definition}

See Figure \ref{fig:rothe} for an example of a Rothe diagram.

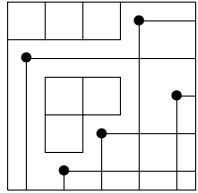
\begin{figure}[h] 
	\begin{tikzpicture}
	\draw (0,0)--(2.5,0)--(2.5,2.5)--(0,2.5)--(0,0);
	\draw[]
	(2.5,1.75)
	-- (0.25,1.75) node {$\bullet$}
	-- (0.25,0);
	\draw[]
	(2.5,0.25)
	-- (0.75,0.25) node {$\bullet$}
	-- (0.75,0);
	\draw[]
	(2.5,0.75)
	-- (1.25,0.75) node {$\bullet$}
	-- (1.25,0); 
	\draw[]
	(2.5,2.25)
	-- (1.75,2.25) node {$\bullet$}
	-- (1.75,0);
	\draw[]
	(2.5,1.25)
	-- (2.25,1.25) node {$\bullet$}
	-- (2.25,0);
	
	\draw (0,2)--(0.5,2)--(0.5,2.5);
	\draw (0.5,2)--(1,2)--(1,2.5);
	\draw (1,2)--(1.5,2)--(1.5,2.5);
	
	\draw (1,1)--(1.5,1)--(1.5,1.5)--(1,1.5)--(1,1)--(1,0.5)--(0.5,0.5)--(0.5,1)--(1,1);
	\draw (0.5,1)--(0.5,1.5)--(1,1.5);
	
	\end{tikzpicture} \label{fig:rothe}
 \caption{The Rothe diagram of $\pi=41532$
		is $(\{1\}, \{1,3,4\}, \{1,3\}, \emptyset, \emptyset)$.}
	\label{fig:rothe}
\end{figure}


\begin{definition}
	If $\alpha=(\alpha_1,\alpha_2,\ldots)$ is a composition, let $l=\max\{i:\,\alpha_i\neq 0 \}$ and $n=\max\{l,\alpha_1,\ldots, \alpha_l \}$. The \newword{skyline diagram} of $\alpha$ is the diagram $D(\alpha)\subseteq[n]^2$ containing the first $\alpha_i$ boxes in row $i$ for each $i\in[n]$. More specifically, $D(\alpha)=(D(\alpha)_1,\ldots, D(\alpha)_n)$ with $D(\alpha)_j=\{j\leq n:\, \alpha_j\geq j \}$ for each $j$.
\end{definition}

See Figure \ref{fig:sky} for an example of a skyline diagram.

\begin{figure}[h] 
	\begin{tikzpicture}
	\draw (0,0)--(2.5,0)--(2.5,2.5)--(0,2.5)--(0,0); 
	\draw (.5,2.5)--(.5,1); 
	\draw (1,2.5)--(1,1.5); 
	\draw (1.5,2.5)--(1.5,2); 
	\draw (0,2)--(1.5,2); 
	\draw (0,1.5)--(1,1.5); 
	\draw (0,1)--(.5,1);
	\draw (0,.5)--(.5,.5); 
	\draw (.5,.5)--(.5,0);
	
	
	\end{tikzpicture}\caption{The skyline diagram of $\alpha=(3,2,1,0,1)$ is $(\{1,2,3,5\}, \{1,2\}, \{1\}, \emptyset, \emptyset)$.}
	\label{fig:sky}
\end{figure}
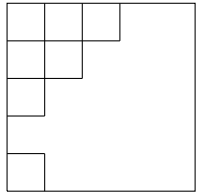

\subsection{Newton polytopes and generalized permutahedra} 
 
If $f$ is a polynomial in a polynomial ring whose variables are indexed by some set $I$,
the \newword{support} of $f$ is the lattice point set in $\mathbb R^I$ consisting of the exponent vectors of monomials with nonzero coefficient in~$f$.
The \newword{Newton polytope} $\Newton(f)\subseteq\mathbb R^I$ is the convex hull of the support of~$f$. Following the definition of \cite{MTY}, we say that a polynomial $f$ has \newword{saturated Newton polytope (SNP)} if every lattice point in $\Newton(f)$ is a vector in the support of $f$. In other words, SNP means that the polynomial is equal to a positively weighted integer point transform of its Newton polytope.

Our main objects of study are the supports of Schubert and  key polynomials. We prove that they have SNP and that their Newton polytopes are generalized permutahedra,   which we define next.  

The standard permutahedron is the polytope in $\mathbb{R}^n$ whose vertices consist of all permutations of the entries of the vector $(0,1,\ldots,n-1)$. A \textbf{generalized permutahedron} is a deformation of the standard permutahedron obtained by translating the vertices in such a way that all edge directions and orientations are preserved (edges are allowed to degenerate to points). Generalized permutahedra are parametrized by 
certain collections of real numbers $\{z_I\}$ indexed by nonempty subsets $I\subseteq [n]$, and have the form
\[P_n^z(\{z_I \})=\left \{ {\bf t}\in\mathbb{R}^n:\, \sum_{i\in I}{t_i}\geq z_I \mbox{ for } I\neq [n], \mbox{ and } \sum_{i=1}^{n}{t_i}=z_{[n]}   \right \}. \]  Postnikov initiated the study of these fascinating polytopes in \cite{beyond}, and they have since been studied extensively. 

The class of generalized permutahedra is closed under Minkowski sums. This follows from \cite[Lemma 2.2]{matroidpolytopes}:
\[P_n^z(\{z_I \})+P_n^z(\{z'_I \})=P_n^z(\{z_I+z'_I \}). \]

\subsection{Schubert matroids}

A \newword{matroid} $M$ is a pair $(E, \B)$ consisting of a finite set $E$ and a nonempty collection of subsets $\B$ of $E$, 
called the \newword{bases} of $M$. $\B$ is required to satisfy the \newword{basis exchange axiom}: 
If $B_1, B_2 \in \B$ and $b_1 \in B_1- B_2$, then there exists $b_2 \in B_2 - B_1$ such that $B_1 - b_1 \cup b_2 \in \B$. By choosing a labeling of the elements of $E$, we can assume $E=[n]$ for some $n$.

Fix positive integers $1 \leq s_1 < \ldots < s_r \leq n$. The sets $\{a_1, \ldots, a_r\}$ with $a_1<\cdots<a_r$ such that $a_1 \leq s_1, \ldots, a_r \leq s_r$ are the bases of a matroid, 
called the \newword{Schubert matroid} $SM_n(s_1, \ldots, s_r)$  \cite[Section 2.4]{OMbook}. 

\subsection{Matroid polytopes}
Given a matroid $M=(E,\mathcal{B})$ with $E=[n]$, the \newword{rank function} of $M$ is the function
\[r_M:2^{E}\to \mathbb{Z}_{\geq 0}\] defined by $r_M(S)=\max\{\#(S\cap B):\, B\in \mathcal{B} \}$. The sets $S\cap B$ where $S\subseteq [n]$ and $B\in\mathcal{B}$ are called the \newword{independent sets} of $M$.

The \newword{matroid polytope} of $M$ is the generalized permutahedron $P(M)$ defined by
\begin{align*}
	P(M)&=P^z_n\left(\{ r_M(E)-r_M(E\backslash I) \}_{I\subseteq E}\right)\\
	&=\left\{ {\bf t}\in\mathbb{R}^n:\, \sum_{i\in I}{t_i}\leq r_M(I) \mbox{ for } I\neq E, \mbox{ and } \sum_{i\in E}{t_i}=r_M(E) \right\}.
\end{align*}
The vertices of $P(M)$ are exactly the indicator vectors of the bases of $M$: if $B\in \mathcal{B}$ is a basis of $M$ and $\zeta^B=(\zeta_1^B,\ldots \zeta_n^B)\in\mathbb{R}^n$ is the vector with $\zeta_i^B=1$ if $i\in B$ and $\zeta_i^B=0$ if $i\notin B$ for each $i$, then 
\[P(M)=\mathrm{Conv}\{\zeta^B:\,B\in\mathcal{B} \}. \]

\subsection{Flagged Weyl modules}
Let $G=\mathrm{GL}(n,\mathbb{C})$ be the group of $n\times n$ invertible matrices over $\mathbb{C}$ and $B$ be the subgroup of $G$ consisting of the $n\times n$ upper-triangular matrices. The flagged Weyl module is a representation $M_D$ of $B$ associated to a diagram $D$, 
whose dual character has been shown in certain cases to be a Schubert polynomial or a key polynomial. 
We use the construction of $M_D$ in terms of determinants given in \cite{Magyar}.

Denote by $Y$ the $n\times n$ matrix with indeterminants $y_{ij}$ in the upper-triangular positions $i\leq j$ and zeros elsewhere. Let $\mathbb{C}[Y]$ be the polynomial ring in the indeterminants $\{y_{ij}\}_{i\leq j}$. Note that $G$ acts on $\mathbb{C}[Y]$ on the right via left translation: if $f({\bf y})\in \mathbb{C}[Y]$, then a matrix $g\in G$ acts on $f$ by $f({\bf y})\cdot g=f(g^{-1}{\bf y})$. For any $R,S\subseteq [n]$, let $Y_R^S$ be the submatrix of $Y$ obtained by restricting to rows $S$ and columns $R$.

For $R,S\subseteq [n]$, we say $R\leq S$ if $\#R=\#S$ and the $k$\/th least element of $R$ does not exceed the $k$\/th least element of $S$ for each $k$. For any diagrams $C=(C_1,\ldots, C_n)$ and $D=(D_1,\ldots, D_n)$, we say $C\leq D$ if $C_j\leq D_j$ for all $j\in[n]$.

\begin{definition}
	For a diagram $D=(D_1,\ldots, D_n)$, the \newword{flagged Weyl module} $M_D$ is defined by
	\[M_D=\mathrm{Span}_\mathbb{C}\left\{\prod_{j=1}^{n}\det\left(Y_{D_j}^{C_j}\right):\, C\leq D \right\}. \]
	$M_D$ is a $B$-module with the action inherited from the action of $B$ on $\mathbb{C}[Y]$. 
\end{definition}
Note that since $Y$ is upper-triangular, the condition $C\leq D$ is technically unncessary since $\det\left(Y_{D_j}^{C_j}\right)=0$ unless $C_j\leq D_j$.

\section{Newton Polytopes of Dual Characters of Flagged Weyl Modules}
\label{sec:schub}

For any $B$-module $N$, the \newword{character} of $N$ is given by  \[\mathrm{char}(N)(x_1,\ldots,x_n)=\mathrm{tr}\left(X:N\to N\right) \] 
where $X$ is the diagonal matrix $\mathrm{diag}(x_1,x_2,\ldots,x_n)$ with diagonal entries $x_1,\ldots,x_n$, and $X$ is viewed as a linear map from $N$ to $N$ via the $B$-action.

Define the \newword{dual character} of $N$ to be the character of the dual module $N^*$:
\begin{align*}
	\mathrm{char}^*(N)(x_1,\ldots,x_n)&=\mathrm{tr}\left(X:N^*\to N^*\right) \\
	&=\mathrm{char}(N)(x_1^{-1},\ldots,x_n^{-1}).
\end{align*}

\begin{theorem}[\cite{KP}] Let $w\in S_n$ be a permutation, $D(w)$ be the Rothe diagram of $w$, and $M_{D(w)}$ be the associated flagged Weyl module. Then, 
	\[\mathfrak{S}_w(x_1,\ldots,x_n) = \mathrm{char}^*M_{D(w)}. \]
\end{theorem}

\begin{theorem}[\cite{keypolynomials}] Let $\alpha$ be a composition, $D(\alpha)$ be the skyline diagram of $\alpha$, and $M_{D(\alpha)}$ be the associated flagged Weyl module. If $l=\max\{i:\,\alpha_i\neq 0 \}$ and $n=\max\{\alpha_1,\ldots,\alpha_l,l \}$, then 

\[\kappa_\alpha(x_1,\ldots,x_n) = \mathrm{char}^*M_{D(\alpha)}. \]
\end{theorem}

\begin{definition}
	For a diagram $D\subseteq [n]^2$, let $\chi_D=\chi_D(x_1,\ldots,x_n)$ be the dual character 
	\[\chi_D=\mathrm{char}^*M_D. \] 
\end{definition}

\begin{theorem}
	\label{thm:Newtonofdualcharacter}
	Let $D=(D_1,\ldots, D_n)$ be a diagram. Then $\chi_D$ has SNP, and the Newton polytope of $\chi_D$ is the Minkowski sum of matroid polytopes
	\[\mathrm{Newton}(\chi_D)=\sum_{j=1}^{n}P(SM_n(D_j)). \]
	In particular, $\Newton(\chi_D)$ is a generalized permutahedron.
\end{theorem}

\begin{proof}
	Denote by $X$ the diagonal matrix $\mathrm{diag}(x_1,x_2,\ldots,x_n)$. First, note that $y_{ij}$ is an eigenvector of $X$ with eigenvalue $x_i^{-1}$. Take a diagram $C=(C_1,\ldots,C_n)$ with $C\leq D$. Then, the element $\prod_{j=1}^{n}\det\left(Y_{D_j}^{C_j}\right)$ is an eigenvector of $X$ with eigenvalue 
	\[\prod_{j=1}^{n}\prod_{i\in C_j}x_i^{-1}.\]	
	Since $M_D$ is spanned by elements $\prod_{j=1}^{n}\det\left(Y_{D_j}^{C_j}\right)$ and each is an eigenvector of $D$, the monomials appearing in the dual character $\chi_D$ are exactly 
	\[\left\{\prod_{j=1}^{n}\prod_{i\in C_j}x_i:\, C\leq D \right\}. \]
	
	For a diagram $C=(C_1,\ldots, C_n)$, define a vector $\xi^C=(\xi_1^C,\ldots, \xi_n^C)$ by setting $\xi_i^C=\#\{j:\,i\in C_j\}$ for each $i$. The exponent vector of $\prod_{j=1}^{n}\prod_{i\in C_j}x_i$ is exactly $\xi^C$, so the support of $\chi_D$ is precisely the set $\left\{\xi^C:\, C\leq D \right\}$.

	However, for each $j\in[n]$, the sets $S\subseteq [n]$ with $S\leq D_j$ are exactly the bases of the Schubert matroid $SM_n(D_j)$. In particular, choosing a diagram $C\leq D$ is equivalent to picking a basis $C_j$ of $SM_n(D_j)$ for each $j\in[n]$. If $\zeta^{C_j}$ is the indicator vector of $C_j$, then comparing components shows
	\[\xi^C=\sum_{j=1}^{n}\zeta^{C_j}. \]
	This shows that each vector $\xi^C$ is a sum consisting of a vertex from each matroid polytope $P(SM_n(D_j))$ for $j\in [n]$. Conversely, given any sum $\sum_{j=1}^{n}\zeta^{B_j}$ of a vertex $\zeta^{B_j}$ from each $P(SM_n(D_j))$, let $C=(B_1,\ldots,B_n)$. Since each $B_j$ is a basis of $SM_n(D_j)$, $C\leq D$. Thus, $\xi^C=\sum_{j=1}^{n}\zeta^{C_j}$ is in the support of $\chi_D$.\\
	
	\noindent Consequently, 
	\[\Newton(\chi_D)= \sum_{j=1}^{n}P(SM_n(D_j)).  \]
	In particular, we have shown that each lattice point of $\Newton(\chi_D)$ corresponds to a sum consisting of a vertex from each $P(SM_N(D_j))$. It follows from \cite[Corollary 46.2c]{Schrijver} that each lattice point in this Minkowski sum is the sum of a lattice point in each summand $P(SM(D_j))$. However, the only lattice points in a matroid polytope are its vertices. Hence, $\chi_D$ has SNP.
\end{proof}

\begin{corollary}\label{thm:Schubert SNP}
	The support of any Schubert polynomial $\mathfrak{S}_{w}$ or key polynomial $\kappa_\alpha$
	equals the set of lattice points a generalized permutahedron.
\end{corollary}

This confirms Conjectures 3.10 and 5.1 of \cite{MTY}, namely that key polynomials and Schubert polynomials have SNP. We now confirm Conjectures 3.9 and 5.13 of \cite{MTY}, which give a conjectural inequality description for the Newton polytopes of Schubert and key polynomials. We state this description and match it to the Minkowski sum description proven in Theorem \ref{thm:Newtonofdualcharacter}.

Let $D\subseteq [n]^2$ be any diagram with columns $D_j=\{i:\, (i,j)\in D \}$ for $j\in[n]$. Let $I\subseteq [n]$ be a set of row indices and $j\in[n]$ a column index. Construct a string $\mathrm{word}_{j,I}(D)$ by reading column $j$ of the $n$ by $n$ grid from top to bottom and recording
\begin{itemize}
	\item $($ if $(i,j)\notin D$ and $i\in I$;
	\item $)$ if $(i,j)\in D$ and $i\notin I$;
	\item $\star$ if $(i,j)\in D$ and $i\in I$.
\end{itemize}
Let $\theta_D^j(I)=\#\mbox{paired }()\mbox{'s in } \mathrm{word}_{j,I}(D) + \#\star\mbox{'s in }\mathrm{word}_{j,I}(D)$, and set 
\[\theta_D(I)=\sum_{j=1}^{n}\theta_D^j(I) .\]

\begin{definition}[\cite{MTY}]
	For any diagram $D\subseteq [n]^2$, define the Schubitope $\mathcal{S}_D$ by
	\[\mathcal{S}_D=\left\{(\alpha_1,\,\ldots,\,\alpha_n)\in \mathbb{R}_{\geq 0}^n:\,\sum_{i=1}^{n}\alpha_i=\#D \mbox{ and } \sum_{i\in I}\alpha_i\leq \theta_D(I) \mbox{ for all } I\subseteq [n] \right\}. \]
\end{definition}

\begin{theorem} \label{thm:schubi}
	Let $D$ be a diagram $D\subseteq [n]^2$ with columns $D_j=\left\{i:\,(i,j)\in D \right\}$ for each $j\in [n]$. The  Schubitope $\mathcal{S}_D$ equals the Minkowski sum of matroid polytopes
	\[
	\mathcal{S}_D=\sum_{j=1}^{n}{P\left(SM_n\left(D_j\right)\right)}.
	\]
\end{theorem}
\begin{proof}
	Let $r_j$ be the rank function of the matroid $SM_n(D_j)$. By \cite[Lemma 2.2]{matroidpolytopes}, the Minkowski sum $\sum_{j=1}^{n}{P\left(SM_n\left(D_j\right)\right)}$ equals
	\[
	\left\{(\alpha_1,\,\ldots,\,\alpha_n)\in \mathbb{R}_{\geq 0}^n:\,\sum_{i\in[n]}\alpha_i=\sum_{j=1}^{n}r_j([n]) \mbox{ and } \sum_{i\in I}\alpha_i\leq \sum_{j=1}^{n} r_j(I)  \mbox{ for all } I\subseteq [n] \right\}.
	\]
	Thus, it is sufficient to prove that $\theta_D^j(I)=r_j(I)$ for each $j\in[n]$ and $I\subseteq[n]$. Fix $I$ and $j$, and let $\mathrm{word}_{j,I}(D)$ have $p$ paired $()$'s and $q$ $\star$'s.
	
	First, note that $D_j$ is a basis of $SM_n(D_j)$. Let $B$ be any basis of $SM_n(D_j)$ and pick elements $r_1$ and $r_2$ with $r_1\notin B$, $r_2\in B$, and $r_1<r_2$. Consider the set $B'=B\backslash\{r_2\}\cup\{r_1\}$. Then $B'\leq B\leq D_j$, so $B'$ is also a basis of $SM_n(D_j)$. Using this observation, we build a decreasing sequence of bases $D_j\geq B_1 \geq \cdots \geq B_p$.
	
	Order the set of paired $()$'s in $\mathrm{word}_{j,I}(D)$ from 1 to $p$. For the first pair, we get two grid squares $(r_1,\,j)$ and $(r_2,\,j)$ with $r_1<r_2$, $r_1\in I\backslash D_j$, and $r_2\in D_j\backslash I$. Define $B_1$ to be the basis $D_j\backslash \{r_2\}\cup\{r_1\}$. 
	
	Inductively, the $i$th set of paired $()$'s in $\mathrm{word}_{j,I}(D)$ gives two grid squares $(r_1,\,j)$ and $(r_2,\,j)$ with $r_1<r_2$, $r_1\in I\backslash B_{i-1}$, and $r_2\in (B_{i-1}\cap D_j)\backslash I$. Define $B_i$ to be the basis $B_{i-1}\backslash \{r_2\}\cup\{r_1\}$.
	
	By construction, $\#(I\cap B_p)=p+\#(I\cap D_j)=p+q$. The proof will be complete of we can show $I\cap B_p$ is a maximal independent subset of $I$. If not, there is some $k\in I\backslash B_p$ and $l\in (B_p\cap D_j)\backslash I$ such that $B_p\backslash\{l\}\cup\{k\}$ is a basis. If $k<l$, then $k$ and $l$ correspond to a $()$ in $\mathrm{word}_{j,S}(D)$, so $k\in B_p$ already, a contradiction. If $k>l$, then in $\mathrm{word}_{j,S}(D)$, $k$ and $l$ correspond to a subword $)($ where neither parenthesis was paired. Then, the position of $l$ in $B_p$ is the same as the original position of $l$ in $D_j$, since it cannot have been changed by any of the swaps. In this case, $k>l$ implies $B\backslash \{l\}\cup\{k\}$ is not a basis. 
\end{proof}

Theorem \ref{thm:schubi} confirms Conjectures 3.9 and 5.13 of \cite{MTY}.
 
\section*{Acknowledgements} We thank Bal\'azs Elek and Allen Knutson for inspiring conversations.

\bibliography{biblio-kir}
\bibliographystyle{plain}

\end{document}